\documentclass[12pt,reqno]{amsart}
\usepackage{graphicx}
\usepackage{psfrag}
\usepackage{pxfonts}
\usepackage{mathrsfs}
\usepackage{color}
\usepackage{amsmath,amssymb}

\numberwithin{equation}{section}

\newtheorem{theorem}{Theorem}[section]

\newtheorem{corollary}[theorem]{Corollary}
\newtheorem{proposition}[theorem]{Proposition}
\newtheorem{lemma}[theorem]{Lemma}
\newtheorem{definition}[theorem]{Definition}

\theoremstyle{remark}
\newtheorem{remark}[theorem]{Remark}

\begin{document}

\thanks{}

\author{F. Micena}
\address{Departamento de Matem\'atica,
  IM-UFAL Macei\'{o}-AL, Brazil.}
\email{fpmicena@gmail.com}


\renewcommand{\subjclassname}{\textup{2000} Mathematics Subject Classification}

\date{\today}

\setcounter{tocdepth}{2}

\title{New Derived from Anosov Diffeomorphisms   with pathological center foliation. }
\maketitle
\begin{abstract}
In this paper we focused our study on Derived From Anosov diffeomorphisms (DA  diffeomorphisms ) of the torus $\mathbb{T}^3,$ it is, an absolute partially hyperbolic diffeomorphism on $\mathbb{T}^3$ homotopic to an
Anosov linear automorphism of the $\mathbb{T}^3.$ We can prove that if $f: \mathbb{T}^3 \rightarrow \mathbb{T}^3 $ is a volume preserving DA diffeomorphism homotopic to linear Anosov $A,$ such that the center Lyapunov exponent satisfies $\lambda^c_f(x) > \lambda^c_A > 0,$ with $x $ belongs to a positive volume set,  then the center foliation of $f$ is non absolutely continuous.  We construct a new open  class $U$ of non Anosov and volume preserving DA diffeomorphisms, satisfying the property $\lambda^c_f(x) > \lambda^c_A > 0$ for $m-$almost everywhere $x \in \mathbb{T}^3.$ Particularly for every $f \in U,$ the center foliation of $f$ is non absolutely continuous.
\end{abstract}

\section{Introduction and Statements of the Results}

Let $M$ be a $C^{\infty}$ riemannian closed (compact, connected and boundaryless)  manifold. A $C^1-$diffeomorphism $f: M \rightarrow M$ is called a partially hyperbolic diffeomorphism if the tangent bundle $TM$ admits a $Df$ invariant tangent decomposition $TM =  E^s \oplus E^c \oplus E^u$ such that all unitary vectors $v^{\sigma} \in E^{\sigma }_x, \sigma \in \{s,c,u\}$ for every $x \in M$ satisfy:

$$ ||D_x f v^s || < ||D_x f v^c || < ||D_x f v^u ||,$$

moreover

$$||D_x f v^s || < 1 \;\mbox{and}\; ||D_x f v^u || > 1 $$

We say that a partially hyperbolic diffeomorphism $f$ is an absolute partially hyperbolic diffeomorphism if

$$||D_x f v^s || < ||D_y f v^c || < ||D_z f v^u ||$$

for every $x,y,z \in M$ and $v^s, v^c, v^u$ are unitary vectors in $E^s_x, E^c_y, E^u_z$ respectively.

From now, in this paper, when we require partial hyperbolicity, we mean absolute partially hyperbolicity and all diffeomorphisms considered are at least $C^1.$ We go to denote by $PH_m(\mathbb{T}^3)$ the set of all partially hyperbolic diffeomorphisms which preserve the volume form $m.$

In partially hyperbolic context it is well known that the sub-bundles $E^s, E^u,$ respectively the stable and ubstable sub-bunbles are uniquely integrable to invariant foliations $\mathcal{F}^s, \mathcal{F}^u$ respectively (see \cite{HPS}). The sub-bundle $E^c $ is not necessarily uniquely integrable to a invariant foliation $\mathcal{F}^c,$ in fact, in \cite{HHU2} the authors provide examples of (non absolute) partially hyperbolic diffeomorphisms, such that $E^c $ is not  uniquely integrable.

\begin{definition} A partially hyperbolic diffeomorphism $f: M \rightarrow M$ is called dynamically coherent if $E^{cs} := E^c \oplus E^s$ and $E^{cu} := E^c \oplus E^u$  are uniquely integrable to invariant foliations $\mathcal{F}^{cs}$ and  $\mathcal{F}^{cu},$ respectively the center stable and center unstable foliations. Particularly $E^c$ is uniquely integrable to the center foliaiton $\mathcal{F}^{c},$ which is obtained by the intersection $\mathcal{F}^{cs} \cap \mathcal{F}^{cu}.$
\end{definition}

When $M = \mathbb{T}^3,$ Brin-Burago- Ivanov, in \cite{BBI}, shown that:

\begin{theorem}
\cite{BBI} All partially hyperbolic diffeomorphisms $f: \mathbb{T}^3 \rightarrow \mathbb{T}^3 $ are dynamically coherent.
\end{theorem}

Every diffeomorphism of the torus $\mathbb{T}^n$ induces an automorphism of the fundamental group and there exists a unique linear diffeomorphism $f_{\ast}$ which induces the same automorphism on $\pi_1(\mathbb{T}^n).$ The diffeomorphism $f_{\ast}$ is called linearization of $f.$ In this paper we study relations between the center Lyapunov exponent of $f$ and the center Lyapunov exponents of $f_{\ast}$ under absolute continuity of the center foliation of $f.$  The relations founded will allow construct a new open class of diffeomorphisms  in $PH_m(\mathbb{T}^3)$ which the center foliation is pathological, i.e,  non absolutely continuous.

\begin{definition} Let $f: \mathbb{T}^3 \rightarrow \mathbb{T}^3 $ be a partially hyperbolic diffeomorphism, $f$ is called a Derived from Anosov (DA) diffeomorphism if its linearization $f_{\ast}: \mathbb{T}^3 \rightarrow \mathbb{T}^3 $ is a linear Anosov automorphism.
\end{definition}

By \cite{H}, given $f: \mathbb{T}^3 \rightarrow \mathbb{T}^3 $ be a partially hyperbolic diffeomorphism, then $f_{\ast}$ is also partially hyperbolic, moreover there is a homeomorphism $h: \mathbb{T}^3 \rightarrow \mathbb{T}^3 $ that carries center leaves of $f_{\ast}$ to corresponding center leaves of
$f,$ it is,

$$ h(\mathcal{F}^c_{f_{\ast}}(x)) = \mathcal{F}^c_f(h(x)),$$

where  $\mathcal{F}^c_g(y)$ is the center leaf of $g$ through $y.$

Particularly the center foliation a DA diffeomorphism of the $\mathbb{T}^3$ is non compact.

\subsection{Lyapunov Exponents}

Lyapunov exponents are important constants and measure the asymptotic behavior of dynamics in tangent space level. Let $f: M \rightarrow M$ be a measure preserving diffeomorphism. Then by the Oseledec theorem, for almost every $x \in M $ and any $v \in T_x M $ the following limit exists:

$$\lim_{n \rightarrow +\infty} \frac{1}{n} \log ||Df^n(x) \cdot v ||$$

and it is equal to one of the Lyapunov exponents of $f.$ For a volume preserving partially hyperbolic of $\mathbb{T}^3,$ which is the main object of the study in this paper, we get a full Lebesgue measure subset  $\mathcal{R}$   such that for each $x \in \mathcal{R}:$

$$\lim_{n \rightarrow +\infty} \frac{1}{n} \log ||Df^n(x) \cdot v^{\sigma} || = \lambda^{\sigma}_f(x),$$

where $\sigma \in \{s,c,u\}$ and $v^{\sigma} \in E^{\sigma}\setminus \{0\}.$

A result of Hammerlindl-Ures states an important dichotomy between ergodicity and the center Lyapunov exponent equal to zero.

\begin{theorem} \cite{HU} \label{teoHU} Suppose that $f: \mathbb{T}^3 \rightarrow \mathbb{T}^3 $ is a conservative $C^2-$DA diffeomorphism. If $f$ is not ergodic, then  the center Lyapunov exponent is zero almost everywhere.
\end{theorem}

\subsection{Absolute Continuity} It is known that the foliations $\mathcal{F}^s$ and $\mathcal{F}^u$ of a conservative $C^{1+\alpha}$ diffeomorphism $f: M \rightarrow M$ satisfies a property called absolute continuity. Absolute continuity is a fundamental tool in the Hopf argument in the proof of ergodicity of $C^{1+\alpha}$ conservative Anosov diffeomorphisms. Roughly speaking a foliation $\mathcal{F}$ of $M$ is absolutely continuous if satisfies: Given a set $Z \subset M,$ such that $Z$ intersects the leaf $\mathcal{F}(x)$ on a zero measure  set of the leaf, with $x$ along a full Lebesgue set of $M,$ then $Z$ is a zero measure set of $M.$
More precisely we write:

\begin{definition} We say that a foliation $\mathcal{F}$ of $M$ is absolutely continuous if given any $\mathcal{F}-$foliated box $B$ and a Lebesgue measurable set $Z,$ such that $Leb_{\mathcal{F}(x)\cap B} (\mathcal{F}(x)\cap Z) = 0,$ for $m_B-$ almost everywhere $x \in B,$ then $m_B(Z) = 0.$ Here $m_B $ denotes the Lebesgue measure on $B$ and $Leb_{\mathcal{F}(x)\cap B} $ is the Lebesgue measure of the submanifold $\mathcal{F}(x)$ restricted to $B.$

$$Leb_{\mathcal{F}(x) \cap B  } ((\mathcal{F}(x) \cap B) \cap Z ) = 0, \; m_B- a.e. x \in B  \Rightarrow m_B(Z) = 0. $$
\end{definition}

It means that if $P$ is such that $m_B(P) > 0,$ then there are a measurable subset $B' \subset B,$ such that $m_B(B') > 0$ and  $Leb_{\mathcal{F}(x)\cap B}(\mathcal{F}(x) \cap Z) > 0$ for every $x \in B'.$

The study of absolute continuity of the center foliation started with Ma\~{n}\'{e},  that noted a interesting relation between absolute continuity and the center Lyapunov exponent.
The Ma\~{n}\'{e}'s argument can be explained as the following theorem:

\begin{theorem} Let $f: M \rightarrow M$ be a  partially hyperbolic, dynamcally coherent such that $\dim(E^c) = 1$ and $\mathcal{F}^c$ is a compact foliation. Suppose $f$ preserves a volume form $m$ on $M,$  and the set of $x \in M$ such that   $\lambda^c_f(x) > 0$ has positive volume.  Then $\mathcal{F}^c$ is non absolutely continuous.
\end{theorem}

\begin{proof} Denote by $P$ the set of $x \in M$ such that   $\lambda^c_f(x) > 0.$ Consider $\Lambda_{k,l, n } =\{ x \in P | \; ||Df^j(x)|E^c|| \geq e^{\frac{j}{k}},\; \mbox{for every} \;  j \geq l , \; \mbox{and} \; |\mathcal{F}^c(x)| < n   \} ,$ here $|\mathcal{F}^c(x)|$ denotes the size of the center leaf $\mathcal{F}^c(x)$ through $x. $

We have $P = \displaystyle\bigcup_{k,l, n \in \mathbb{N}} \Lambda_{k,l, n }, $ in particular there are $k_0, l_0, n_0$ such that $m(\Lambda_{k_0,l_0, n_0 }) > 0.$  Supposing  that $\mathcal{F^c}$ is an absolutely continuous foliation, there is a center leaf $\mathcal{F}^c(x), $ such that it intersects  $\Lambda_{k_0,l_0, n_0 }$ on a positive Lebesgue measure set of the leaf. By Poincar\'{e}-recurrence Theorem, the  point $x $ can be chosen a recurrent point, particularly there is a subsequence $n_k$ such that $f^{n_k}(x) \in \Lambda_{k_0,l_0, n_0 }, $ and it implies that the size $|\mathcal{F}^c({f^{n_k}(x)})| < n_0.$

On the other hand, we denote

$$\alpha = Leb_{\mathcal{F}^c(x)}(\mathcal{F}^c(x) \cap \Lambda_{k_0,l_0, n_0 } ),$$

so, if $j \geq k_0$ we have $|\mathcal{F}(f^j(x))| \geq \alpha\cdot e^{\frac{j}{k_0}} \rightarrow +\infty$ when $j \rightarrow +\infty, $ and it contradicts $|\mathcal{F}^c({f^{n_k}(x)})| < n_0$ for a subsequence $n_k.$ Consequentely all one dimensional compact and absolutely continuous center foliation implies that $\lambda^c_f(x) = 0,$ for $m-$ almost everywhere $x \in M.$
\end{proof}

\begin{remark} Katok exhibits an example of a volume preserving partially hyperbolic diffeomorphism $f: \mathbb{T}^3 \rightarrow \mathbb{T}^3 $ such that $\mathcal{F}^c$ is compact, non absolutely continuous and $\lambda^c_f(x)=0$ for $m-$ a.e. $x \in \mathbb{T}^3 . $ See \cite{HaPe} and citations therein.
\end{remark}

Ruelle-Wilkinson and Pesin-Hirayama generalized the Ma\~{n}\'{e} argument, we can state these results in a unique theorem as following:

\begin{theorem} \cite{HP}, \cite{RW} Consider a dynamically coherent partially hyperbolic
diffeomorphism $f$ whose center leaves are fibers of a (continuous) fiber
bundle. Assume that the all center Lyapunov exponents are negative (or positive)
then the conditional measures of $\mu$ on the leaves of the center foliation are atomic
with $p, p \geq 1,$ atoms of equal weight on each leaf.

\end{theorem}

In the non compact case Saghin-Xia in \cite{SX} shown that:

 \begin{theorem}\label{teoSX} Let $g \in Diff_m(\mathbb{T}^d)$ close to a linear Anosov automorphism $L:\mathbb{T}^d\rightarrow \mathbb{T}^d $ with $\lambda^c_L > 0.$ If

$$ \displaystyle \int_{\mathbb{T}^d} \log(||Dg| E^c_g||) dm > \lambda^c_L,$$

then the foliation $\mathcal{F}^c_g$ is non absolutely continuous.
\end{theorem}

In the sense of theorem \ref{teoSX} we are able to prove its generalization for DA diffeomorphisms.

\begin{theorem} [Theorem A]
Let $f: \mathbb{T}^3 \rightarrow \mathbb{T}^3 $ a $m-$ preserving DA diffeomorphism with linearization $A: \mathbb{T}^3 \rightarrow \mathbb{T}^3. $ Suppose that there is a measurable set $P,$ with $m(P) > 0,$ such that  $\lambda^c_f(x) > 0$  for every $x \in P.$ If $\mathcal{F}^c_f$ is absolutely continuous, then $0 \leq \lambda^c_f(x) \leq \lambda^c_A, $ for $m-a.e.$ point $x \in \mathbb{T}^3.$
\end{theorem}

In the Anosov case, Gogolev in \cite{Go} describes completely the question of the absolute continuity of the center foliation of $C^{1 + \alpha}$ conservative Anosov diffeomorphisms of $\mathbb{T}^3.$

\begin{theorem}\cite{Go} Let $L$ be an automorphism of $\mathbb{T}^3$ with three distinct
Lyapunov exponents $\lambda^s_L < 0 < \lambda^c_L < \lambda^u_L.$ Let $U_L$ the open set (in the volume preserving setting) of all $C^{1+ \alpha}$ volume preserving Anosov diffeomorphism homotopic to $L.$ Let $f \in U_L,$ then $\mathcal{F}^c_f$ is absolutely continuous if, and only if, $\lambda^u(p) = \lambda^u(q)$ for every periodic points $p$ and $q.$
\end{theorem}

In \cite{PTV} the authors proved study the desintegration of the volume along to $\mathcal{F}^c_f,$ where $f$ is a DA diffeomorphism of $\mathbb{T}^3 $ homotopic to a liner hyperbolic authomorphism $A,$ when $\lambda^c_f \cdot \lambda^c_A < 0.$

\begin{theorem} Consider $A$ a linear Anosov  automorphism of $\mathbb{T}^3$ with three distinct
Lyapunov exponents $\lambda^s_A <  \lambda^c_A < \lambda^u_A.$ Let $f : \mathbb{T}^3 \rightarrow \mathbb{T}^3$
be volume preserving
DA diffeomorphism (homotopic to A). Assume that $f$ is partially hyperbolic,
volume preserving and ergodic. Also assume that $\lambda^c_f \cdot \lambda^c_A < 0,$ then the disintegration
of volume along center leaves of $f $ is atomic and in fact there is just
one atom per leaf.
\end{theorem}

 In this paper we treated the case  $\lambda^c_f \cdot \lambda^c_A > 0,$   non considered in \cite{PTV}. Denote by $DA_m(\mathbb{T}^3)$ the set of all $m$ preserving DA diffeomorphism of $\mathbb{T}^3$ and $\mathcal{A}(\mathbb{T}^3)$ the set of all Anosov diffeomorphism of $\mathbb{T}^3.$ In the case $\lambda^c_f \cdot \lambda^c_A > 0,$ we can prove:

\begin{theorem}[Theorem B] There is an open set $U \subset DA_m(\mathbb{T}^3) \setminus \overline{\mathcal{A}(\mathbb{T}^3)}, $ such that for any $f \in U$ has the same linearization $A$ and  $\lambda^c_f(x) > \lambda^c_A > 0,$  for $m$ almost everywhere $x \in \mathbb{T}^3.$ Particularly $\mathcal{F}^c_f$ is non absolutely continuous for every $f \in U.$
\end{theorem}

The proof of Theorem B consists to combine two different types of perturbations in some steps as follows:

\begin{enumerate}
\item Firstly, using the linear maps introduced in \cite{PT}, we have linear Anosov linear automorphisms, with center Lyapunov exponent arbitrarily close to zero.

\item After, by a small Baraviera-Bonatti perturbation (see \cite{BB}), we increase a little the center Lyapunov exponent.

\item Using the conservative version of Franks lemma (see \cite{BDP}), we modify the stable index of a fixed point, but yet preserving in this step the increment of the center Lyapunov exponent obtained in the previous step. The perturbation here is made carefully, such that it remains partially hyperbolic.

 \item The neighborhood $U$ requiered in the Theorem B will be an open set in $PH_m(\mathbb{T}^3)$ around the diffeomorphism obtained in the previous step, or an open ball around an stably ergodic perturnation of diffeomorphisms obtained in the previous step.
\end{enumerate}

\section{Proof of Theorem A}

Before to prove the Theorem A, let us give some ingredients necessary to  the proof.

\begin{definition} A foliation $\mathcal{F}$ of a closed manifold $M$ is called quasi-isometric if there is a constant $Q > 0,$ such that in the cover level $\tilde{M}$ we have:

$$ d_{\mathcal{\widetilde{W}}}(x, y) \leq Q\cdot d_{\widetilde{M}}(x,y) + Q, $$

for every $x,y$ points in the same lifted leaf $\widetilde{\mathcal{W}},$ of $\widetilde{\mathcal{F}},$ where $\widetilde{\mathcal{F}}$ denotes the lift of $\mathcal{F}$ on $\tilde{M}.$   Here $d_{\widetilde{\mathcal{W}}}$ denotes the riemannian metric on $\widetilde{\mathcal{W}}$ and $d_{\widetilde{M}}$ is a riemannian metric of $\widetilde{M}.$

\end{definition}

\begin{theorem} \cite{BBI}, \cite{H} Let $f: \mathbb{T}^3 \rightarrow \mathbb{T}^3$ be a partially hyperbolic diffeomorphism, then $\mathcal{F}^s, \mathcal{F}^c$ and $\mathcal{F}^u$ are quasi-isometric foliations.
\end{theorem}

The quasi isometry of the invariant foliations of a partially hyperbolic diffeomorphism implies some consequences of the geometry of the leafs in large scale as stated in the next lemmas.

\begin{lemma}\label{lemmaH} \cite{H} Let $f: \mathbb{T}^3 \rightarrow \mathbb{T}^3$ be a partially hyperbolic diffeomorphism with linearization $A.$ For each $k \in \mathbb{Z}$ and $C > 1$ there is an $M > 0$ such that for all $x, y \in \mathbb{R}^3,$

$$||x - y || \geq M  \Rightarrow \frac{1}{C} <  \frac{||\tilde{f}^k(x) - \tilde{f}^k(y)||}{|| \tilde{A}^k(x) - \tilde{A}^k(y)|| } < C, $$

where $\tilde{f}: \mathbb{R}^3 \rightarrow \mathbb{R}^3 $ denotes the lift of $f$ to $ \mathbb{R}^3.$
\end{lemma}

\begin{lemma}\label{lemmaMT}\cite{MT} Let $f: \mathbb{T}^3 \rightarrow \mathbb{T}^3$ be a partially hyperbolic diffeomorphism with linearization  $A.$ For each $n \in \mathbb{Z}$ and $\varepsilon > 0 $ there exists $M > 0$ such that for every $x, y $  in the same lifted leaf of $\tilde{\mathcal{F}}^{\sigma}, \sigma \in \{s,c,u\}$ we have

$$ ||x - y|| \geq M \Rightarrow (1 + \varepsilon)^{-1} e^{n \lambda^{\sigma}_A} ||x- y|| \leq || \tilde{A}^n(x) - \tilde{A}^n(y)|| \leq (1 + \varepsilon) e^{n \lambda^{\sigma}_A}||x- y||, $$

where $\lambda^{\sigma}_A$ is the Lyapunov exponent of corresponding to $E^{\sigma}_A$ and $\sigma \in \{s,c,u\}.$

\end{lemma}

Combining the previous lemmas we can state.

\begin{lemma}\label{joint} Let $f: \mathbb{T}^3 \rightarrow  \mathbb{T}^3 $ be a DA diffeomorphism, such that $ f_{\ast} = A$ and $\lambda^c_A > 0.$ Given  $k = 1,$ and $C = (1+ \varepsilon),$ for a small $\varepsilon > 0,$ consider $M > 0 $ satisfying the lemmas \ref{lemmaH} and \ref{lemmaMT}. If $x, y \in \mathbb{R}^3$ on the same center leaf of $\tilde{f}$, such that $||x-y|| > M,$ then

$$ || \tilde{f}^n(x) - \tilde{f}^n(y)|| \leq (1 + \varepsilon)^{2n} e^{n\lambda^c_A} ||x - y||, \; \mbox{for all} \; n \geq 1.$$
\end{lemma}

\begin{proof} The proof is by induction on $n.$ If $n = 1 ,$ by the lemma $\ref{lemmaH}$ we have
$$||\tilde{f}(x) - \tilde{f}(y)|| \leq  (1 +  \varepsilon) \cdot||\tilde{A}(x) - \tilde{A}(y)||,$$

and by lemma \ref{lemmaMT}

$$||\tilde{A}(x) - \tilde{A}(y)|| \leq (1 + \varepsilon) e^{\lambda^c_A} ||x - y||$$

combining the two last expressions we have

$$||\tilde{f}(x) - \tilde{f}(y)|| \leq (1 + \varepsilon)^2 e^{\lambda^c_A} ||x - y||.$$

It is important to note that$||\tilde{f}(x) - \tilde{f}(y)|| \geq ||x - y|| \geq M.$ In fact

$$||\tilde{f}(x) - \tilde{f}(y)|| \geq
(1+ \varepsilon)^{-1} \cdot ||\tilde{A}x - \tilde{A}y||, \; \mbox{by lemma \ref{lemmaH} and} $$

$$||\tilde{A}x - \tilde{A}y|| \geq (1 + \varepsilon)^{-1}e ^{\lambda^c_A} \cdot ||x - y||,\;  \mbox{by lemma \ref{lemmaMT} },$$

thus

$$||\tilde{f}(x) - \tilde{f}(y)|| \geq (1 + \varepsilon)^2 e^{\lambda^c_A}||x - y||,$$

if $\varepsilon > 0 $ is enough small, we have $||\tilde{f}(x) - \tilde{f}(y)|| \geq ||x - y||.$  It allows us apply the argument (for $k = 1$) above to $\tilde{f}(x) , \tilde{f}(y)$ in the same center leaf in $\mathbb{R}^3.$

By induction suppose that $||\tilde{f}^k(x) - \tilde{f}^k(y)|| \leq (1 + \varepsilon)^{2k} e^{k\lambda^c_A} ||x - y||,$ for some $k \geq 1.$ Like described above $||\tilde{f}^k(x) - \tilde{f}^k(y)|| \geq M,$ then we apply the argument presented in the case $k = 1$ to the points $\tilde{f}^k(x) , \tilde{f}^k(y)$ in the same center leaf in $\mathbb{R}^3.$

Thus,

$$||\tilde{f}^{k+1}(x) - \tilde{f}^{k+1}(y)|| = ||\tilde{f}(\tilde{f}^k(x)) - \tilde{f}(\tilde{f}^k(y))|| \leq (1 + \varepsilon)^{2} e^{\lambda^c_A} || \tilde{f}^k(x) -\tilde{f}^k(y)|| $$

since

$$||\tilde{f}^k(x) - \tilde{f}^k(y)|| \leq (1 + \varepsilon)^{2k} e^{k\lambda^c_A} ||x - y||,$$

we have

$$||\tilde{f}^{k+ 1}(x) - \tilde{f}^{k+1}(y)|| \leq (1 + \varepsilon)^{2(k + 1)} e^{(k + 1)\lambda^c_A} ||x - y||,$$

the induction is completed.

\end{proof}

\subsection{Proof of the Theorem A}

\begin{proof}
Under the assumptions of the  Theorem A, we must to prove two facts:

\begin{enumerate}
\item $\lambda^c_f(x) > 0, x \in P$ with $m(P) > 0 \Rightarrow \lambda^c_A > 0,$
\item $ 0 \leq \lambda^c_f(x) \leq \lambda^c_A,$ for $m$ a.e. $x \in \mathbb{T}^3.$
\end{enumerate}

For each $\frac{1}{d} > 0,$ with $d \in \mathbb{N}$ consider the set

$$P_d := \left\{x \in P| \lambda^c_f(x) \geq \frac{1}{d} \right\}.$$

Since $P = \bigcup_{d = 1}^{+\infty}P_d,$ and $m(P) > 0,$ then there exists $d > 0$ such that $m(P_d ) > 0.$

Now, for each $n \in \mathbb{N}$ consider

$$P_{d, n} = \{ x \in P_d|\; ||Df^k(x)| E^c_f(x)|| \geq e^{\frac{k}{2d}}, \forall k \geq n\},$$

then, there exists $N > 1$ such that $m(P_{d,N}) > 0.$

For each $x \in \mathbb{T}^3,$ choose  $B_x$ be an $\mathcal{F}^c_f-$foliated box, such that $x$ lies in the interior of $B_x.$ Since $\mathbb{T}^3$ is compact, there are $x_1, \ldots, x_j$ such that $\{B_{x_i}\}_{i = 1}^j$ covers $\mathbb{T}^3.$ Thus $m(P_{d,N} \cap B_{x_i}) > 0$ for some $1 \leq i \leq j.$

Since $\mathcal{F}^c_f$ is an absolutely continuous foliation, there exists $p \in B_{x_i}$ such that  the component of $\mathcal{F}^c_f(p)$ in $B_{x_i}$ intersects $P_{d, N}$ in a positive riemannian measure set of the leaf. Denote this component by  the center segment $[a,b]^c.$

Suppose that $\alpha \cdot |[a, b]^c| = Leb_{\mathcal{F}^c_f(p)} ([a, b]^c\cap P_{d,N}), \alpha > 0,  $ then in the  lifting $[\tilde{a}, \tilde{b}]^c$ of $[a, b]^c$ we have

$$|\tilde{f}^n([\tilde{a}, \tilde{b}]^c)| \geq \alpha \cdot e^{\frac{n}{2d}} |[\tilde{a}, \tilde{b}]^c|, n \geq N, $$

where $|[\tilde{a}, \tilde{b}]^c|$ denotes the length of the segment $[\tilde{a}, \tilde{b}]^c . $  In particular $|\tilde{f}^n([\tilde{a}, \tilde{b}]^c)| \rightarrow +\infty,$

and since $\mathcal{F}^c_f$ is quasi-isometric, we have

\begin{equation}
|| \tilde{f}^n(\tilde{a}) - \tilde{f}^n(\tilde{b})  || \geq \frac{1}{Q}(|\tilde{f}^n([\tilde{a}, \tilde{b}]^c)| - Q) \rightarrow + \infty. \label{infty}
\end{equation}

Let $k , \varepsilon > 0,$ $C > 0$ and $M > 0$ as in the lemmas $\ref{lemmaH}$ and $\ref{lemmaMT}$ and consider  $x_n, y_n$ the extremes of $\tilde{f}^n([a, b]^c),$  by quasi isometry of $\mathcal{F}^c_f$ there exists $n_0$ such that  if $n \geq n_0$ then $||x_n - y_n || > M.$ Then combining the lemmas lemmas $\ref{lemmaH}$ and $\ref{lemmaMT}$,  we have

$$(C(1 + \varepsilon))^{-1}\leq \frac{||\tilde{f}^{k}(x_n ) - \tilde{f}^{k}(y_n )||}{ e^{k \lambda^c_A}  ||x_n - y_n|| } \leq C(1 + \varepsilon),$$

by the equation $(\ref{infty})$  we have $||\tilde{f}^{k}(x_n ) - \tilde{f}^{k}(y_n )|| \rightarrow +\infty ,$ follows that $e^{\lambda^c_A} > 1,$ Thus $\lambda^c_A > 0.$

It remains to prove that $0 \leq \lambda^c_f \leq \lambda^c_A.$

Suppose by contradiction that $\lambda^c_f(x) > \lambda^c_A$ on a Lebesgue positive set $\Lambda$ and $\mathcal{F}^c_f$ is absolutely continuous.  Choose a small $\delta > 0$ and $\varepsilon > 0$ such that

$$  m(\{x \in \Lambda | \; e^{\lambda^c_f(x)} \geq  (1 + 10\varepsilon)^2 e^{\lambda^c_A + \delta} \})> 0.$$

Define

$$\Lambda_{\delta, d} = \{x \in \Lambda| \; ||Df^n(x)| E^c_f(x)|| \geq (1+ 10\varepsilon)^2 e^{\lambda^c_A + \delta}, \forall n \geq d\}.$$

Since $m(\{x \in \Lambda | \; e^{\lambda^c_f(x)} \geq  (1 + 10\varepsilon)^2 e^{\lambda^c_A + \delta}\}) > 0,$ then there is some $d$ such that $m(\Lambda_{\delta, d}) > 0.$

Now for each $x \in \mathbb{T}^3$ consider $B_x$ an open $\mathcal{F}^c_f-$foliated box, such that $x \in B_x$ and for each $y \in B_x,$ if $[a, b]^c$ is the center segment in $B_x$ containing $y,$ then its lifting denoted by $[\tilde{a}, \tilde{b}]^c $ is such that $||\tilde{a} - \tilde{b}|| > M .$  Where $M$ satisfies the lemma \ref{joint} with $C = (1 + \varepsilon)$ and $k= 1.$
By compactness of $\mathbb{T}^3,$ there are $B_{x_1}, \ldots, B_{x_j}$ a finite subcover of $\mathbb{T}^3.$ Since $\mathcal{F}^c_f$ is absolutely continuous, then there is $ 1 \leq i \leq j$ such that

$$Leb_{[a,b]^c} ( [a, b]^c \cap \Lambda_{\delta, d} ) > 0,$$

where $[a,b]^c$ is a center connected component in $B_{x_i}.$

Let $\alpha > 0$ be such that $Leb_{[a,b]^c} ( [a, b]^c \cap \Lambda_{\delta, d} ) = \alpha \cdot |[a, b]^c|.$

Thus, the length

$$|\tilde{f}^n([\tilde{a}, \tilde{b}]^c)| \geq \alpha \cdot (1 + 10 \varepsilon)^{2n}e^{n (\lambda^c_A + \delta)}|[\tilde{a}, \tilde{b}]^c|, \forall n \geq d.$$

particularly using quasi isometry of the foliation $\mathcal{F}^c_f$ we have

\begin{equation}\label{final}
||\tilde{f }^n(\tilde{a}) - \tilde{f }^n(\tilde{b})|| \geq \frac{\alpha}{Q} ((1 + 10 \varepsilon)^{2n}e^{n( \lambda^c_A + \delta)}|[\tilde{a}, \tilde{b}]^c| - Q).
\end{equation}

Applying the lemma \ref{joint} to $[\tilde{a}, \tilde{b}]^c,$ with $C  = ( 1 + \varepsilon ),$ we obtain

\begin{equation} \label{final2}
|| \tilde{f}^n(\tilde{a}) - \tilde{f}^n(\tilde{b})|| \leq (1 + \varepsilon)^{2n} e^{n\lambda^c_A} ||\tilde{a} - \tilde{b}||, \; \mbox{for all} \; n \geq 1.
\end{equation}

The inequalities $(\ref{final})$ and $(\ref{final2})$ contradict one each other, then $\mathcal{F}^c_f$ can not be absolutely continuous under the assumptions.

\end{proof}

\begin{corollary} Let $f: \mathbb{T}^3 \rightarrow \mathbb{T}^3 $ a $C^1-$volume preserving DA diffeomorphism. Suppose that there are  Lebesgue measurable sets $P, N,$ with $m(P)\cdot m(N) > 0,$ such that $\lambda^c_f(x) > 0$ for every $x \in P$ and $\lambda^c_f(x) < 0$ for every $x \in N,$ then $\mathcal{F}^c_f$ is not absolutely continuous.
\end{corollary}

\begin{proof} Suppose that $\mathcal{F}^c_f$ is absolutely continuous and there are Lebesgue measurable sets $P, N,$ with $m(P)\cdot m(N) > 0,$ such that $\lambda^c_f(x) > 0$ for every $x \in P$ and $\lambda^c_f(x) < 0$ for every $x \in N.$

$$\lambda^c_f(x) > 0 \; \mbox{ for every } \; \in P \Rightarrow \lambda^c_A > 0,$$

$$\lambda^c_f(x) < 0 \; \mbox{ for every } \; \in N \Rightarrow \lambda^c_A < 0.$$

The last implications above are contradictories, it concludes the proof.

\end{proof}

\begin{remark} The statement of the corollary above makes sense only in the $C^1$ setting, if $f$ is $C^r, r \geq 2$ the ergodicity in the theorem \ref{teoHU}  obstructs the existence of  Lebesgue measurable sets $P, N,$ with $m(P)\cdot m(N) > 0,$ such that $\lambda^c_f(x) > 0$ for every $x \in P$ and $\lambda^c_f(x) < 0$ for every $x \in N.$
\end{remark}

In \cite{U} the author shown:

\begin{theorem}\label{teoU} \cite{U} Let $f: \mathbb{T}^3 \rightarrow \mathbb{T}^3 $ a $C^{1+ \alpha}-$ DA diffeomorphism, homotopic to $A
,$ with $\lambda^c_A > 0.$ Then there is a unique maximizing entropy measure $\mu$ for $f,$ moreover $(f, \mu)$ and $(A, m)$ are isomorphic and  $\lambda^c_f(x) \geq \lambda^c_A,$ for $\mu$ almost everywhere $x \in  \mathbb{T}^3.$
\end{theorem}

Relying in the previous theorem we can prove the following corollary.

\begin{corollary} Let $f: \mathbb{T}^3 \rightarrow \mathbb{T}^3 $ a  $C^{1+ \alpha}-$ volume preserving DA diffeomorphism, homotopic to $A,$ with $\lambda^c_A > 0.$ Suppose that $m$ is the maximal entropy measure and $\mathcal{F}^c_f$ is absolutely continuous, then $\lambda^{\sigma}_f(x) = \lambda^{\sigma}_A, \sigma \in \{s,c,u\}$ for $m$ almost everywhere $x \in \mathbb{T}^3.$
\end{corollary}

\begin{proof} By the  theorem \ref{teoU} we have $\lambda^c_f(x) \geq \lambda^c_A > 0$ for $m-$ almost everywhere $x \in \mathbb{T}^3.$ On the other hand, since $\lambda^c_f(x ) > 0$ for $m$ almost everywhere $x \in \mathbb{T}^3$ and $\mathcal{F}^c_f$ is absolutely continuous,  then  by  Theorem A we have  $\lambda^c_f \leq \lambda^c_A $  for almost everywhere $x \in \mathbb{T}^3.$ So $\lambda^c_f(x) = \lambda^c_A$  for almost everywhere $x \in \mathbb{T}^3.$

By the theorem $\ref{teoU}$ the systems $(f, m)$ and $(A,m)$ are isomorphic, then

$$h_m (f) = h_m (A) = \lambda^c_A + \lambda^c_A,$$

by Pesin's formula

$$ h_m(f) = \lambda^c_A + \int_{\mathbb{T}^3} \lambda^u_f dm =  \lambda^c_A + \lambda^c_A = h_m(A), $$

since, by \cite{MT} we have $\lambda^u_f(x) \leq \lambda^u_A,$ for for $m$ almost everywhere $x \in \mathbb{T}^3,$ it jointly with the expression above imply  that $m(\{x \in \mathbb{T}^3| \; \lambda^u_f(x) < \lambda^u_A\}) = 0.$ Then $\lambda^u_f(x) = \lambda^u_A,$ for $m$ almost everywhere $x \in \mathbb{T}^3.$ Consequently $\lambda^s_f(x) = \lambda^s_A$ for $m$ almost everywhere $x \in \mathbb{T}^3.$

\end{proof}

\section{Proof of Theorem B}

For to give the proof of Theorem B,
firstly we construct examples of $f \in  DA_m(\mathbb{T}^3)\setminus \overline{\mathcal{A}}(\mathbb{T}^3)$, such that $\lambda^c_f > \lambda^c_A > 0$ for
$m-$ a.e.,  where $A$ is the linearization of $f.$
The open sets required will be neighborhoods of these
examples or neighborhoods of perturbations of the constructed examples. For
the construction we need recall some results.

\begin{proposition}\label{franks}[Conservative Franks Lemma, proposition 7.4 of \cite{BDP}] Consider a conservative diffeomorphism $f$ and a finite $f-$ invariant set
$E.$ Assume that $B$ is a conservative $\varepsilon-$ perturbation of $Df$ along $E.$ Then for every
neighbourhood $V$ of $E$ there is a conservative diffeomorphism $g$ arbitrarily $C^1-$ close to
$f$  coinciding with $f$ on $E$ and out of $V$, and such that $Dg$ is equal to $B$ on $E.$

\end{proposition}

\begin{proposition}\label{babo}\cite{BB} Let $(M,m)$ be a compact manifold endowed with a $C^r$ volume
form,$r \geq 2.$ Let $f$  be a $C^1$ and $m-$ preserving diffeomorphisms of $M,$  admitting a dominated
partially hyperbolic splitting $TM= E^s
\oplus E^c \oplus E^u.$ Then there are arbitrarily $C^1-$close
and $m-$preserving perturbation $g$ of $f,$ such that

$$\displaystyle \int_M \log(||Dg|E^c_g||)dm > \displaystyle \int_M \log(||Df|E^c_f||)dm.$$

\end{proposition}

For to begin the construction, for each $n \geq 1,$ as in \cite{PT}, we  consider $L_n: \mathbb{T}^3 \rightarrow \mathbb{T}^3 $ the Anosov automorphism the $\mathbb{T}^3$ induced by the matrices

$$L_n = \left[\begin{array}{ccc}
0 & 0 & 1\\
0 & 1 & -1\\
-1 & -1 & n
\end{array}\right].$$

We go to considerate the linear automorphism induced by $B_n = L_n^{-1}.$ By \cite{PT} $B_n$ have three distinct  eigenvalues $\beta^u(n), \beta^c(n), \beta^s(n),$ satisfying

\begin{equation}
\frac{\beta^u(n)}{n} \rightarrow 1, \label{unstable}
\end{equation}

\begin{equation}
\beta^c(n) \rightarrow 1^+, \label{center}
\end{equation}

\begin{equation}
n \cdot \beta^s(n) \rightarrow 1, \label{stable}
\end{equation}

calling $E^u_n, E^c_n, E^s_n$ respectively the eigenspaces corresponding to $\beta^u(n), \beta^c(n), \beta^s(n),$ by \cite{PT} we have:

\begin{equation}
 E^u_n \rightarrow \langle 1,0,0\rangle,\label{unstable1}
\end{equation}

\begin{equation}
E^c_n \rightarrow \langle 0,1,0\rangle, \label{center1}
\end{equation}

\begin{equation}
E^s_n \rightarrow \langle 0,0,1\rangle, \label{stable1}
\end{equation}

 where $\langle v\rangle$ denotes de real subspace spanned by $v.$ We go to denote $\beta^c(n) = 1 + \alpha_n,$ such that $\alpha_n \downarrow 0. $ The center exponent $\lambda^c_n : = \lambda^c_{B_n} = \log( 1 + \alpha_n) \downarrow 0.$

By the proposition \ref{babo}, and the structural stability of Anosov diffeomorphisms, consider $g_n$ a small perturbation of $L_n,$ an Anosov diffeomorphism such that $B_n$ and $g_n$ are $\alpha_n$ close in the $C^1$ topology and

\begin{equation}
\log(1 + \alpha_n) = \lambda^c_n < \displaystyle\int_{\mathbb{T}^3} \log (||Dg_n| E^c_{g_n}||) dm < \log(1 + 2\alpha_n).
\end{equation}

Let $p_n$ be a fixed point for $g_n.$ Consider a system $\{V_{n_j}\}_{j = 0}^{+\infty}$ of  small open balls centered in $p_n.$ The neighborhoods $V_{n_j} $ are constructed after. Using the proposition \ref{franks} we can obtain a $C^1-$perturbation $g_{n_j}$ of $g_n$ satisfying:

\begin{enumerate}
\item $p_n$ is a fixed point for $g_{n_j},$
\item $g_{n_j} = g_n$ out of $V_{n_j},$
\item $Dg_{n_j}(p_n)| E^u_{g_{n_j}}(p_n) =  Dg_{n}(p_n)| E^u_{g_{n}}(p_n),$
\item $Dg_{n_j}(p_n) $ and $Dg_n(p_n)$ have the same eigenspaces and the same orientation on corresponding eigenspaces,
\item $|| Dg_{n_j}(p_n)| E^c_{g_{n_j}}(p_n)|| = 1 - \alpha_n,$
\item $Dg_{n_j}(p_n)| E^s_{g_{n_j}}(p_n)$ is taken coherently with $g_{n_j}$ being $m-$ preserving.
\end{enumerate}

\subsection{The chosen of the open system $\{V_{n_j}\}_{j = 1}^{+\infty}$}\label{neigh}  Fixed $\varepsilon > 0  $ a small number and $V_{n_0} = B(p_n, \varepsilon)$ the open ball centered in $p_n$ with radius $\varepsilon.$ Since $g_n$ is Anosov, it  is possible to choose $0< \varepsilon_1 < \varepsilon_0,$ such that, if $x \in \mathbb{T}^3 \setminus V_{n_0}$ then

$$g_n^k(x) \in B(p_n, \varepsilon_1) \Rightarrow |k| > 1,$$

we take $V_{n_1} = B(p_n, \varepsilon_1). $ By proposition \ref{franks} we have

$$g_n^k(x) \in B(p_n, \varepsilon_1) \Rightarrow |k| > 1.$$

Suppose that we have defined $V_{n_0}, V_{n_1}, \ldots, V_{n_j},$ we define recursively $V_{n_{j+1}} = B(p_n, \varepsilon_{j+ 1}),$ with $0 < \varepsilon_{j+1} < \varepsilon_{j},$ such that, if $x \in \mathbb{T}^3 \setminus V_{n_j}$ then

$$g_n^k(x) \in B(p_n, \varepsilon_{j+1}) \Rightarrow |k| > j + 1,$$

and consequentely, by proposition \ref{franks}

$$g_{n_{j + 1}}^k(x) \in B(p_n, \varepsilon_{j+1}) \Rightarrow |k| > j + 1.$$

In the construction we require that $0<diam(V_{n_j})< \frac{1}{j + 1}$ and clearly the diameters $diam(V_{n_j}) \rightarrow 0,$ when $j \rightarrow + \infty.$

\subsection{The diffeomorphisms $g_{n_j}$ remains partially hyperbolic}

\begin{lemma}\label{partially}
For $n$ large, the diffeomorphisms  $g_{n_j}$ are partially hyperbolic for any $j \geq 0.$
\end{lemma}

\begin{proof}

Each $g_{n_j}$ are $\varepsilon_n : = 100 \alpha_n $ close to $B_n,$ in the $C^1$ topology. In particular the matrices $Dg_{n_j}(x)$ and $B_n$ are $\varepsilon_n$ close, it implies that, there is a constant $C > 0$ such that the corresponding terms of the matrices $B_n$ and $g_{n_j}(x)$ are
$C \varepsilon_n$ close, for any $x \in \mathbb{T}^3$. Without loss generality, we go to consider $C = 1. $

Since $Dg_{n_j}(x)$ are $\varepsilon_n$ close to $B_n,$
 we have:

 \begin{enumerate}

 \item The restriction $B_n| E^u_n$ expands by a uniform constant bigger than $\frac{n}{2}, $ then $Dg_{n_j}(x)| E^u_n$ expands by a constant bigger than $\frac{n}{2}.$
 \item The restriction $B_n|E^c_n$ are close to the identity, then $Dg_{n_j}(x)| E^c_n$ are close to the identity.

 \item The restriction $B_n|E^s_n$ contracts by a uniform constant less that $\frac{2}{n},$ it implies that  $Dg_{n_j}(x)| E^s_n$ contracts by a uniform constant less than $\frac{2}{n} + \varepsilon_n.$

 \end{enumerate}

The items above implies that (from canonical Jordan form) there are subspaces  $F^u_{n_j}(x)$ and $F^s_{n_j}(x)$ invariant for $Dg_{n_j}(x),$ such that
 
$$ Dg_{n_j}(x)|F^u_{n_j}(x) \; \mbox{is uniform expanding with constant bigger than} \;\frac{n}{2}$$

$$ Dg_{n_j}(x)|F^s_{n_j}(x) \; \mbox{is uniform contracting with constant less than} \;\frac{2}{n} + \varepsilon_n,$$

Let $w^u_{n_j}(x) \in F^u_{n_j}(x) \setminus\{0\}$ and $w^s_{n_j}(x) \in F^u_{n_j}(x) \setminus\{0\},$ be unit vectors. Also, for each  $n,$ consider $\{e^u_n, e^c_n, e^s_n\}$ unitary eigenvectors of $B_n.$

Since $Dg_{n_j}(x)$ contracts uniformly $w^s_{n_j}(x)$ by a constant less than $\frac{2}{n} + \varepsilon_n,$ then the projection of $w^s_{n_j}(x)$ on $E^u_n$ has size of the order of $\frac{1}{n} \cdot (\frac{2}{n} + \varepsilon_n).$  A similar argument allows to claim that the projection of $w^s_{n_j}(x)$ on $E^c_n$ has size of the order of $ (\frac{2}{n} + \varepsilon_n).$ In particular when $n \rightarrow + \infty $ we have

\begin{equation}
\angle (E^s_n, F^s_{n_j}(x)) \rightarrow 0. \label{angleS}
\end{equation}

Analogously, taking the inverses we have

\begin{equation}
\angle (E^u_n, F^u_{n_j}(x)) \rightarrow 0. \label{angleU}
\end{equation}

Since $E^u_n , E^s_n$ are one dimensional subspaces, by the equations $(\ref{angleS})$ and $(\ref{angleU})$ we conclude that $\dim F^u_{n_j}(x) = \dim F^s_{n_j}(x) = 1 $ for every $x \in \mathbb{T}^3,$ any $j \geq 1$ when $n$ is enough large.

Fixed a constant $0 < \theta < 1$ and denote by $P^{\sigma} $ the projection of a vector on $E^{\sigma}_n, \sigma \in \{s,c,u\},$ we have

$$||P^u(Dg_{n_j} (x) \cdot e^u_n)|| \geq \frac{n}{2},  $$

for $n$ large it comes from $(\ref{angleU}),$ moreover

$$||P^s(Dg_{n_j} (x) \cdot e^s_n)|| < 1 \; \mbox{and} \; ||P^c(Dg_{n_j} (x) \cdot e^c_n)|| < 2,   $$

it is true by $(\ref{angleS})$ and the fact $Dg_{n_j}(x) | E^c_n$ is close to the identity.

Then  for $n$ large the cone $C^u_{n_j}(x, \theta) =\{ a\cdot e^u_n + b\cdot e^c_n + k \cdot e^s_n \in T_x \mathbb{T}^3 | \;\; |b| + |k| \leq \theta \cdot |a|\}$ is sent in $C^u_{n_j}(g_{n_j}(x), \theta), $ by $Dg_{n_j}(x),$ it means

$$Dg_{n_j}(x) \cdot C^u_{n_j}(x, \theta) \subset  C^u_{n_j}(g_{n_j}(x), \theta).$$

Analogously the inverse $[Dg_{n_j}(x)]^{-1}$ sent $C^s_{n_j}(g_{n_j}(x), \theta) =\{ a\cdot e^u_n + b\cdot e^c_n + k \cdot e^s_n \in T_x \mathbb{T}^3 | \;\; |b| + |a| \leq \theta \cdot |k|\}$ in $C^{s}_{n_j} (x, \theta).$

Define

$$C^{cu}_{n_j} (x, \theta) = \{ a\cdot e^u_n + b\cdot e^c_n + k \cdot e^s_n \in T_x \mathbb{T}^3 | \;\; |k| \leq \theta \cdot (|a| + |b|) \} $$

$$ C^{cs}_{n_j} (x, \theta) = \{ a\cdot e^u_n + b\cdot e^c_n + k \cdot e^s_n \in T_x \mathbb{T}^3 | \;\; |a| \leq \theta \cdot (|b| + |k|) \} .$$

A similar argument to the one above allows us conclude that

$$ Dg_{n_j}(x)\cdot C^{cu}_{n_j} (x, \theta) \subset C^{cu}_{n_j} (g_{n_j}(x), \theta) $$

and

$$ [Dg_{n_j}(x)]^{-1}\cdot C^{cs}_{n_j} (g_{n_j}(x), \theta) \subset C^{cs}_{n_j} (x, \theta), $$

thus using the characterizations of partial hyperbolicity by cones (see \cite{HaPe}), $g_{n_j}$ is partially hyperbolic for any $j$ when $n$ is arbitrarily large.

\end{proof}

\begin{remark} Note that $g_{n_j}$ is a partially hyperbolic diffeomorphims, but it is not an Anosov diffeomorphism. If fact, the dimension of the local stable manifold of $g_{n_j}$ on $p_n$ is equal two, on the other hand, every periodic point of its linearization $B_n$ has one dimensional stable maninfold. If $g_{n_j}$ was Anosov, then $B_n$ and $g_{n_j}$ would be conjugated one each other, and consequently the dimension of the local stable manifolds of periodic points would must to coincide.
\end{remark}

\subsection{Convergences }

Fixed $n$ large for each $j \geq 0$ and $x \in \mathbb{T}^3 \setminus V_{n_j}$ define

$$N_{j+1}(x) = \min\{ |k| \; | \; g_{n_{j+1}}^k(x) V_{n_{j+1}}\},$$

like the subsection \ref{neigh} we have $N_{j+1}(x) > j+1,$ for every $x \in \mathbb{T}^3 \setminus V_{n_j},$ by the construction of invariant directions using cones, we have:

$$ E^{\sigma}_{g_{n_j}}(x) \rightarrow  E^{\sigma}_{g_{n}}(x), \sigma \in \{s,c,u\}, $$

since $g_n = g_{n_j}$ out of $V_{n_j},$ $diam(V_{n_j}) \rightarrow 0$ and $N_j(x) \rightarrow +\infty.$

By dominated convergence, we have

\begin{equation}
\int_{\mathbb{T}^3 \setminus V_{n_j}} \log(|| Dg_{n_{j+1}}| E^c_{g_{n_{j+1}}}||)dm \rightarrow \int_{\mathbb{T}^3}\log(|| Dg_{n}(x) || ) dm > \lambda^c_n > 0. \label{integral1}
\end{equation}

Since $n$ is fixed $||Dg_{n_j}(x)| E^c_{g_{n_j}} ||$ is bounded, as the diameter $diam(V_{n_j}) \rightarrow 0,$ we have:

\begin{equation}
\int_{V_{n_j}} \log(|| Dg_{n_{j+1}}| E^c_{g_{n_{j+1}}}||) dm \rightarrow 0 \label{integral2}.
\end{equation}

Jointing the equations $(\ref{integral1})$ and $(\ref{integral2})$ we have

\begin{equation}
\int_{\mathbb{T}^3 } \log(|| Dg_{n_{j+1}}| E^c_{g_{n_{j+1}}}||)dm \rightarrow \int_{\mathbb{T}^3}\log(|| Dg_{n}(x) || ) dm > \lambda^c_n > 0. \label{integral3}
\end{equation}

Since $g_{n_j}$ and $B_n$ are $100 \alpha_n $ close in the $C^1$ topology, we have $g_{n_j}$ is homotopic to $B_n$ for every $j$ when $n$ is large enough such that $100 \alpha_n < \frac{1}{5}$ and $g_{n_j}$ being partially hyperbolic for every $j$ as in the lemma \ref{partially}.

\subsection{Conclusion of the proof of Theorem B}

Take a volume preserving partially hyperbolic diffeomorphism $g_{n_j}$ satisfying

$$\int_{\mathbb{T}^3 } \log(|| Dg_{n_{j+1}}| E^c_{g_{n_{j+1}}}||)dm  > \lambda^c_n > 0, $$

it is possible by the expression $(\ref{integral3}).$

If $g_{n_j} \in \partial (\overline{\mathcal{A}}(\mathbb{T}^3)),$ perturb $g_{n_j}$ to a $f_n \in DA_m(\mathbb{T}^3) \setminus \overline{\mathcal{A}}(\mathbb{T}^3)$ a stably ergodic partially diffeomorphism (it is possible by results in \cite{HHU}), such that $\lambda^c_{f_n} > \lambda^c_n > 0.$

Now consider $U \subset DA_m(\mathbb{T}^3) \setminus \overline{\mathcal{A}}(\mathbb{T}^3) $ a small neighborhood of $f_n,$ if $U$ is take a suitable neighborhood, then $\lambda^c_f > \lambda^c_n > 0$ with  $f$ homotopic to $B_n$ for every $f \in U.$

When $g_{n_j} \in  DA_m(\mathbb{T}^3) \setminus \overline{\mathcal{A}}(\mathbb{T}^3), $ we can apply the same argument above, in the case that  $g_{n_j}$ is not stably ergodic.


\end{document}